\documentclass[12pt]{amsart}
\usepackage{amsfonts, amssymb, amsmath, amsthm, eucal, latexsym, array}
\usepackage{fullpage}
\usepackage{verbatim}

\usepackage{color}
\definecolor{darkblue}{rgb}{0,0,.5}
\definecolor{darkgreen}{rgb}{.2,0.5,.2}
\usepackage[colorlinks=true,linkcolor=darkblue,urlcolor=violet,citecolor=magenta]{hyperref}

\numberwithin{equation}{section}

\input{cyracc.def}

\newtheorem{thm}{Theorem}
\newtheorem{lm}{Lemma}[section]
\newtheorem{conj}[lm]{Conjecture}
\newtheorem{cl}[lm]{Corollary}
\newtheorem{prop}[lm]{Proposition}

\theoremstyle{remark}

\newtheorem{ex}[lm]{Example}
\newtheorem{rmk}[lm]{Remark}

\theoremstyle{definition}

\newcommand{\gt}{\mathfrak}

\newcommand{\rk}{\mathrm{rk\,}}

\newcommand {\cS}{{\mathcal S}}

\newcommand {\mK}{{\mathbb C}}

\renewcommand{\le}{\leqslant}
\renewcommand{\ge}{\geqslant}

\newfam\eusfam%
\font\euszw=eusm10 scaled 1200%
\font\eusac=eusm7 scaled 1200%
\font\eusacc=eusm7 scaled 1000%
\textfont\eusfam=\euszw\scriptfont\eusfam=\eusac%
\scriptscriptfont\eusfam=\eusacc%
\newcommand{\U}{{\mathcal U}}

\newcommand{\bth}{\begin{thm}}
\renewcommand{\eth}{\end{thm}}
\newcommand{\bpr}{\begin{prop}}
\newcommand{\epr}{\end{prop}}
\newcommand{\ble}{\begin{lm}}
\newcommand{\ele}{\end{lm}}
\newcommand{\bco}{\begin{cl}}
\newcommand{\eco}{\end{cl}}
\newcommand{\bex}{\begin{ex}}
\newcommand{\eex}{\end{ex}}
\newcommand{\bre}{\begin{rmk}}
\newcommand{\ere}{\end{rmk}}
\newcommand{\bcj}{\begin{conj}}
\newcommand{\ecj}{\end{conj}}

\newcommand{\q}{\mathfrak{q}}

\newcommand{\bal}{\begin{aligned}}
\newcommand{\eal}{\end{aligned}}
\newcommand{\beq}{\begin{equation}}
\newcommand{\eeq}{\end{equation}}
\newcommand{\ben}{\begin{equation*}}
\newcommand{\een}{\end{equation*}}

\newcommand{\bpf}{\begin{proof}}
\newcommand{\epf}{\end{proof}}

\def\beql#1{\begin{equation}\label{#1}}

\begin{document}
\hfill {\scriptsize  June 25, 2018} 
\vskip1ex

\title[The Jacobian of $t$-shifted invariants]{A combinatorial identity  for the 
Jacobian of $t$-shifted invariants}
\author[O.\,Yakimova]{Oksana Yakimova}
\address[O.\,Yakimova]{Universit\"at zu K\"oln,
Mathematisches Institut, Weyertal 86-90, 50931 K\"oln, Deutschland}
\email{yakimova.oksana@uni-koeln.de}
\thanks{This research  is supported by  
the Heisenberg-Stipendium of the DFG}
\begin{abstract}
Let $\gt g$ be a simple Lie algebra. There are classical formulas for 
the Jacobians of the generating invariants of the Weyl group of $\gt g$ and of the images under the
Harich-Chandra projection of the generators of $\mathcal{ZU}(\gt g)$. 
We present a modification of these formulas related to Takiff Lie algebras.  
\end{abstract}
\maketitle

\section*{Introduction}

Let $\gt g$ be a simple complex Lie algebra, $\gt h\subset \gt g$  be a Cartan subalgebra. 
Fix a triangular decomposition $\gt g=\gt n^-\oplus\gt h\oplus\gt n^+$.
Let $\Delta\subset\gt h^*$ be the corresponding root system with 
$\Delta^+\subset \Delta$ being the subset of positive roots. 
Define $\rho=\frac{1}{2}\sum\limits_{\alpha\in\Delta^+} \alpha$. 
Let $W=W(\gt g,\gt h)$ be the Weyl group of $\gt g$. 
Set $n=\rk\gt g$ and let $d_i{-}1$ with $1\le i\le n$ be the exponents of $\gt g$.  
 For $\alpha\in\Delta^+$, let $\{f_{\alpha},h_\alpha,e_\alpha\}\subset\gt g$ be an  $\gt{sl}_2$-triple
 with $e_\alpha\in\gt g_\alpha$. 
 Finally    choose a basis 
$\{h_1,\ldots,h_n\}$ of $\gt h$.

For polynomials $P_1,\ldots, P_n\in\cS(\gt h)\cong \mK[\gt h^*]$, 
the  Jacobian  $J(\{P_i\})$ is defined by the 
property
$$
 d P_1\wedge d  P_2\wedge\ldots\wedge d  P_n=J(\{P_i\}) d h_1\wedge\ldots\wedge d h_{n}. 
$$
If 
$\hat P_1,\ldots, \hat P_n\in\cS(\gt h)^W$ are generating invariants 
(with $\deg  \hat P_i=d_i$), then 
$$J(\{\hat P_i\})=C\prod\limits_{\alpha\in\Delta^+} h_\alpha \ \text{ with } \ C\in \mK,  C\ne 0$$ 
by a classical argument, which is presented, for example, in 
 \cite[Sec.~3.13]{Ref-b}.

Let $\mathcal{ZU}(\gt g)$ denote the centre of the enveloping algebra $\mathcal{U}(\gt g)$. 
Then $\mathcal{ZU}(\gt g)$ has a set $\{{\mathcal P}_i \mid 1\le i\le n\}$ of algebraically  independent 
generators such that 
${\mathcal P}_i\in \U_{d_i}(\gt g)$. 
Let $P_i\in\cS(\gt h)$ be the image of ${\mathcal P}_i$ under the Harish-Chandra projection. 
Then $\hat P_i\in\cS(\gt h)^W$ for $\hat P_i(x)=P_i(x-\rho)$, see e.g. \cite[Sec.~7.4]{Dix}, and 
$$J(\{P_i\})=C\prod\limits_{\alpha\in\Delta^+} (h_\alpha+\rho(h_\alpha)). $$ 

For any complex Lie algebra $\gt l$, let $\varpi\!: \cS(\gt l)\to {\mathcal U}(\gt l)$ be the canonical symmetrsation map. 
Let $\cS(\gt l)^{\gt l}$ denote the ring of symmetric $\gt l$-invariants. 
Since $\varpi$ is an isomorphism of $\gt l$-modules, it provides an isomorphism 
of vector spaces $\cS(\gt l)^{\gt l}\cong \mathcal{ZU}(\gt l)$.
 
Suppose next that ${\mathcal P}_i=\varpi(H_i)$ is the symmetrisation of 
$H_i$ and that $H_i\in\cS(\gt g)^{\gt g}$ is a homogeneous generator of degree $d_i$. 
Let $T\!: \gt g\to \gt g[t]$ be the $\mathbb C$-linear map sending each $x\in\gt g$ to $xt$. 
Then $T$ extends uniquely to the commutative algebras homomorphism 
$$T\!:~\cS(\gt g)\to \cS(\gt g[t]).$$    
Set ${H}_i^{[1]}= T(H_i)$ 
and  ${\mathcal P}_i^{[1]}=\varpi({H}_i^{[1]})$. Here ${\mathcal P}_i^{[1]}\in\U(t\gt g[t])$.  

The  triangular decomposition of $\gt g$ 
extends to 
$\gt g[t]$ as $\gt g[t]=\gt n^-[t]\oplus\gt h[t]\oplus\gt n^+[t]$. 
Let $P_i^{[1]}\in\cS(t\gt h[t])$ be the image 
of ${\mathcal P}_i^{[1]}$ under the 
Harish-Chandra projection. In order to define the Jacobian 
$J(\{P_i^{[1]}\})$ as an element of $\cS(\gt h)$, set at first 
$\partial_{x}(x t^k)=k t^{k-1} $ for every $x\in\gt h$ and every $k\ge 1$, $\partial_{h_i}(h_j t^k)=0$ for $i\ne j$. 
Then the desired formula reads 
$$
J(\{P_i^{[1]}\}) = \det(\partial_{h_j} P_i^{[1]})|_{t=1}.  
$$

\begin{thm}\label{th-eq}
We have the following 
 identity
$$
J(\{P_i^{[1]}\})= C\prod\limits_{\alpha\in\Delta^+} (h_\alpha+\rho(h_\alpha)+1). 
$$
\end{thm}

Our proof of Theorem~\ref{th-eq} interprets the zero set of $J(\{P_i^{[1]}\})$
 in terms of the {\it Takiff Lie algebra} $\gt q=\gt g[u]/(u^2)$ and  
then uses the {\it extremal projector}  associated with $\gt g$, see Section~\ref{sec-proj} for the definition. 

In 1971, Takiff proved that $\cS(\gt q)^{\gt q}$ is a polynomial ring  whose Krull 
dimension equals $2\rk\gt g$~\cite{takiff}. This has started a serious investigation of these 
Lie algebras and their generalisations, see e.g. \cite{k-T} and reference therein.  
Verma modules and an analogue of the Harish-Chandra homomorphism for $\gt q$
were defined and studied in \cite{Geof,Wil}. We remark that $\gt q$-modules appearing in this paper are 
essentially different. 

\section{Several combinatorial formulas}


Keep the notation of the introduction.  In particular, $H_i\in\cS(\gt g)^{\gt g}$ stands for 
a homogeneous generator of degree $d_i$, $P_i$ is the image of ${\mathcal P}_i=\varpi(H_i)$ under 
the Harish-Chandra projection, $\hat P_i\in\cS(\gt h)^W$ is the $({-}\rho)$-shift of $P_i$, i.e., 
$\hat P_i(x)=P_i(x-\rho)$, and $P_i^{[1]}$ is the image of $\varpi(T(H_i))$ under  
the Harish-Chandra projection related to $\gt g[t]$. Let also $P_i^{\circ}$ be the highest degree 
component of $P_i$.  Then $P_i^{\circ}={H_i}|_{\gt h}$. 
By the Chevalley restriction theorem, 
the polynomials $P_i^{\circ}$ with $1\le i\le n$ generate $\cS(\gt h)^W$.
The constant $C$ is fixed by the equality $J(\{P_i^{\circ}\})=C\prod\limits_{\alpha\in\Delta^+} h_\alpha$. It is clear that $J(\{P_i^{\circ}\})=J(\{\hat P_i\})$.

\begin{lm}\label{highest}
The highest degree component of $J(\{P_i^{[1]}\})$ is equal to $C  \prod\limits_{\alpha\in\Delta^+} h_\alpha$. 
\end{lm}
\begin{proof}
The highest degree component of $P_i^{[1]}$ is $T(P_i^{\circ})\in\cS(\gt h t)$. 
Each monomial of $P_i^{[1]}$ is of the form $(x_1t)\ldots (x_{d_i} t)$ with $x_j\in \gt h$ for each $j$. 
By the construction,  $\partial_{h_j} T(P_i^{\circ}) |_{t=1} = \partial_{h_j} P_i^{\circ}$. 
The result follows. 
\end{proof}
In order to prove the next lemma, we need a well-known equality, namely 
$\prod\limits_{i=1}^n d_i=|W|$.

\begin{lm}\label{free}
We have $J(\{P_i^{[1]}\})(0)= C  \prod\limits_{\alpha\in\Delta^+} (\rho(h_\alpha)+1)$. 
\end{lm}
\begin{proof}
Clearly 
$P_i^{[1]}|_{t=1} = P_i$. 
Since $H_i$ is a homogeneous polynomial of degree $d_i$, 
the linear  in $\gt h$ part of  $P_i^{[1]}$ has degree $d_i$ in 
$t$. 
It follows that 
$$
J(\{P_i^{[1]}\})(0)=(d_1\ldots d_n) J(\{P_i\})(0)=|W| \,C \prod\limits_{\alpha\in\Delta^+} \rho(h_\alpha). 
$$
According to a formula of  Kostant:
\begin{equation}\label{KF}
\prod\limits_{\alpha\in\Delta^+} \frac{\rho(h_\alpha)+1}{\rho(h_\alpha)} = |W^\vee|=|W|.
\end{equation}
This completes the proof. 
\end{proof}

\begin{rmk}
The Kostant formula \eqref{KF} is a particular case of another combinatorial statement. 
Let $W(t)=\sum\limits_{w\in W} t^{{\sc l}(w)}$ be the Poincar{\'e} polynomial of $W$.
Then
$$
W^{\vee}(t)=\prod_{\alpha\in\Delta^+} \frac{t^{(\rho,\alpha^{\vee})+1}-1}{t^{(\rho,\alpha^\vee)}-1},
$$
see Equation~(34) in  \cite[Sec.~3.20]{Ref-b}.
Since $\rho(h_\alpha)=(\rho,\alpha^\vee)$, evaluating at $t=1$ one gets exactly Eq.~\eqref{KF}. 
\end{rmk}

\begin{ex} Take $\gt g=\gt{sl}_2$ with the usual basis $\{e,h,f\}$. 
Then $H=H_1=4ef+h^2$, $H^{[1]}=4etft+(ht)^2$, and  
$$
{\mathcal P}^{[1]}=\varpi(H^{[1]})=2etft+2ftet+(ht)^2=(ht)^2+2ht^2+4ftet. 
$$
Therefore $P^{[1]}=P_1^{[1]}=(ht)^2+2ht^2$. Computing the partial derivative and 
evaluating at $t=1$, we obtain $J(\{P^{[1]}\})=2h+4=2(h+2)$. Observe that 
$\rho(h)=1$. 
\end{ex}


\section{Takiff Lie algebras and branching} 

Theorem~\ref{th-eq} can be interpreted as a statement in representation theory of Takiff Lie algebras 
$$\gt q=\gt g{\ltimes}\gt g^{\rm ab}\cong\gt g[u]/(u^2).$$ 
The first factor, the non-Abelian copy of $\gt g$, acts on $\gt g^{\rm ab}=\gt g u$ as a 
subalgebra of $\gt{gl}(\gt g)$. Therefore there is the canonical embedding 
$\q\subset \gt{gl}(\gt g){\ltimes}\gt g^{\rm ab}$. Set $\ell=\dim\gt g+1$. 
In its turn,  $\gt{gl}(\gt g){\ltimes}\gt g^{\rm ab}$ can be realised as a subalgebra of 
$\gt{gl}(\gt g{\oplus}\mK)\cong \gt{gl}_\ell(\mK)$.  The Lie algebra 
$\gt{gl}_\ell(\mK)$ is equipped with the standard triangular decomposition. 
Let $\gt b_\ell\subset \gt{gl}_\ell(\mK)$ be the corresponding positive Borel. 
Recall that we have chosen a triangular decomposition $\gt g=\gt n^-\oplus\gt h\oplus \gt n^+$.
Set $\gt b=\gt h{\oplus}\gt n^+$, $\gt b^-=\gt h{\oplus}\gt n^-$.   
We fix an embedding $\gt{gl}(\gt g){\ltimes}\gt g^{\rm ab}\subset  \gt{gl}(\gt g{\oplus}\mK)$ such that 
$\gt b^-{\ltimes}\gt g^{\rm ab}$ lies in  
the opposite Borel $\gt b_\ell^-$ 
and $\gt b\subset\gt b_\ell$. 

Let $\psi\!: \cS(\gt g)\to \cS(\gt q)$ be a $\mK$-linear map sending 
a monomial $\xi_1\ldots \xi_d$ with $\xi_i\in\gt g$ to 
$$
\sum_{i=1}^d \xi_1\ldots \xi_{i{-}1} (\xi_i u) \xi_{i{+}1}\ldots \xi_d\,. 
$$
Set ${\mathcal R}_i=\varpi(\psi(H_i))$. The elements 
${\mathcal R}_1,\ldots,{\mathcal R}_n\in \U(\gt q)$ are not necessary central. 
If we assume that $d_1=2$, then ${\mathcal R}_1\in \mathcal{ZU}(\gt q)$. 
However, since both maps, $\psi$ and $\varpi$, are homomorphisms of $\gt g$-modules, 
each  ${\mathcal R}_i$ commutes with  $\gt g$. 
Note that  the elements ${\mathcal R}_i$ have degree $1$ in $\gt g u$. 
They are crucial for  further considerations and our next goal is to
relate them to ${\mathcal P}_i^{[1]}\in\U(t\gt g[t])$. 

The map $\psi$ is also well-defined for the tensor algebra of $\gt g$, but not 
for $\U(\gt g)$, because of the following obstacle
$$
\psi(\xi_1\xi_2{-}\xi_2\xi_1)=(\xi_1 u)\xi_2+\xi_1(\xi_2 u)-(\xi_2 u)\xi_1-\xi_2 (\xi_1 u)=
[\xi_1 u,\xi_2]+[\xi_1,\xi_2u]=2[\xi_1,\xi_2]u \ne [\xi_1,\xi_2]u. 
$$
The remedy is to pass to the current algebras $\gt g[t]$ and $\gt q[t]$. 
Let ${\mathcal T}\!: \U(t\gt g[t])\to \U(\gt q[t])$ be a $\mK$-linear map
such that 
$$
\begin{array}{l}
 {\mathcal T}(\xi t^k)=k (\xi u)t^{k{-}1} \ \text{ for each } \ \xi\in\gt g,  \\ 
 {\mathcal T}(ab)= {\mathcal T} (a)b+a {\mathcal T}(b) \ \text{ for all } \ 
 a,b\in\U(t\gt g[t]), \ {\mathcal T} \text{ is a derivation}. 
\end{array} 
$$
Of course, one has to check that ${\mathcal T}$ exists. 


\begin{lm}\label{T}
The map ${\mathcal T}$ is well-defined. 
\end{lm}
\begin{proof}
Take $\xi,\eta\in\gt g$. Then 
$$
 \begin{array}{l}
{\mathcal T}(\xi t^k \eta t^m - \eta t^m\xi t^k)=
k (\xi u) t^{k{-}1} \eta t^m + m \xi t^k (\eta u) t^{m{-}1} - 
m (\eta u) t^{m{-}1} \xi t^{k} - k \eta t^m (\xi u) t^{k{-}1}  = \\
\qquad \enskip = k[\xi u,\eta] t^{k{-}1{+}m} + m[\xi, \eta u] t^{k{+}m{-}1} 
 = (k{+}m) ([\xi,\eta]u) t^{k{+}m{-}1}={\mathcal T}([\xi,\eta]t^{k{+}m}). 
\end{array}
$$
\end{proof}


Now, having the map ${\mathcal T}$, we can state that 
${\mathcal R}_i={\mathcal T}({\mathcal P}_i^{[1]})|_{t=1}$. 

A word of caution, in $\U(\gt b^-)(\gt n^+ u)$ and similar expressions, 
$(\gt n^+ u)$ stands for the subspace $\gt n^+ u\subset\gt g^{\rm ab}$ and {\bf not} 
for an ideal generated by $\gt n^+ u$. The same applies to $(\gt b u)$, $(\gt g u)$, etc.

\begin{lm}  \label{lc}
Let $M_\lambda=\U(\gt b_{\ell}^-)v_\lambda$ with $\lambda\in\mK^{\ell}$ be a Verma module of $\gt{gl}(\gt g{\oplus}\mK)$. Set $\mu=\lambda|_{\gt h}$.  
There exists a non-trivial linear combination ${\mathcal R}=\sum c_i {\mathcal R}_i$ such that 
${\mathcal R}v_\lambda \in  \gt n^-\U(\gt b^-)(\gt n^+ u) v_\lambda$ if and only if 
$J(\{P_i^{[1]}\})(\mu)=0$. 
\end{lm}
\begin{proof}
We have 
$$
{\mathcal P}_i^{[1]}  \in P_i^{[1]} + \gt n^-[t] \U(\gt g[t])\gt n^+[t]. 
$$
Accordingly ${\mathcal R}_i= {\mathcal T}(P_i^{[1]})|_{t=1} + {\mathcal X}$, where 
${\mathcal X}$ is the image of the second summand of ${\mathcal P}_i^{[1]} $.
Let $X=x_1\ldots x_r$ be a monomial appearing in ${\mathcal X}$. 
If $x_r\in \gt n^+$, then $X v_\lambda=0$. Assume that $X v_\lambda\ne 0$. 
Then necessary $x_r\in \gt n^+ u$ and $x_1,\ldots,x_{r-1}\in\gt g$. If $x_i\in\gt n^+$ for some $i\le (r{-}1)$, then 
we replace $X$ by $x_1\ldots x_{i{-}1}[x_i,x_{i+1}\ldots x_r]$. Note that here 
$[x_i,x_r]\in \gt n^+ u$. Applying this procedure as long as possible one 
replaces $X$ by an element of  $\U(\gt b^-)(\gt n^+ u)$ 
without altering $Xv_\lambda$. Since $X$ is an invariant of  $\gt h$, the
new element lies in $\gt n^- \U(\gt b^-)(\gt n^+ u)$. 
Summing up, 
\begin{equation}\label{T-t}
{\mathcal R}_i v_\lambda \in {\mathcal T}(P_i^{[1]})|_{t=1}v_\lambda +  \gt n^- \U(\gt b^-)(\gt n^+ u) v_\lambda.
\end{equation}
Further
$$
{\mathcal T}(P_i^{[1]})|_{t=1}=\sum\limits_{j=1}^n (\partial_{h_j} P_i^{[1]})|_{t=1} h_j u,
$$  
where the partial derivatives are understood in the sense of the introduction. 
Exactly these derivatives have been used in order to define 
 $J(\{P_i^{[1]}\})$. 
Hence $J(\{P_i^{[1]}\})(\mu)=0$ if and only if there is a non-zero vector 
$\bar c=(c_1,\ldots,c_n)$ such that $\sum c_i {\mathcal T}(P_i^{[1]})|_{t=1,\mu}=0$. This shows that if 
 $J(\{P_i^{[1]}\})(\mu)=0$, then ${\mathcal R}v_\lambda \in  \gt n^-\U(\gt b^-)(\gt n^+ u) v_\lambda$.
 
Suppose now that 
${\mathcal R}v_\lambda \in \gt n^-\U(\gt b^-)(\gt n^+ u) v_\lambda\subset \U(\gt n^-)(\gt n^+ u)v_\lambda$. 
Then 
$$
\sum c_i {\mathcal T}(P_i^{[1]})|_{t=1,\mu} v_\lambda \in \U(\gt n^-)(\gt n^+ u) v_\lambda.
$$
Since we are working with a Verma module of $\gt{gl}_{\ell}(\mK)$
and since ${\mathcal T}(P_i^{[1]})|_{t=1,\mu}\in \gt h u \subset \gt n_{\ell}^-$, 
$\gt n^+ u\subset \gt n_{\ell}^-$, we have 
$$\sum c_i {\mathcal T}(P_i^{[1]})|_{t=1,\mu} \in \U(\gt n^-)(\gt n^+ u).$$ 
At the same time  $\gt h u\cap \gt n^+ u=0$. 
Therefore $\sum\limits_{i=1}^n  c_i(\partial_{h_j} P_i^{[1]})|_{t=1,\mu}=0$ for each $j$ and 
thus $J(\{P_i^{[1]}\})(\mu)=0$. 
\end{proof}

For $\gamma\in\gt h^*$, let  $M_{\lambda,\gamma}$ be the corresponding 
weight subspace of $\U(\gt q)v_\lambda\subset M_\lambda$. 
Since $\gt h u \subset \gt n_\ell^-$, either $M_{\lambda,\gamma}=0$ or 
$\dim M_{\lambda,\gamma}=\infty$. We have also $(\gt h u)v_\lambda \ne 0$.  
Because of these facts,  the $\gt q$-modules $M_\lambda$ and 
$\U(\gt q)v_\lambda$ do not fit in the framework of the highest weight theory 
developed in \cite{Geof,Wil}. 
Nevertheless, they may have some nice features.

Lemma~\ref{lc} relates $J(\{P_i^{[1]}\})$ to a property of 
the branching 
$\gt q\downarrow \gt g$ in a particular case of the $\gt q$-module 
$\U(\gt q)v_\lambda$. 
 In order to get a better understanding  of this branching problem, we employ 
a certain projector introduced by 
Asherova, Smirnov, and Tolstoy in \cite{Proj}.

\subsection{The extremal projector} \label{sec-proj}

Recall that  $\{f_{\alpha},h_\alpha,e_\alpha\}\subset \gt g$ is the
$\gt{sl}_2$-triple corresponding to $\alpha\in\Delta^+$.  
 Set 
$$
p_\alpha = 1 + \sum\limits_{k=1}^{\infty} f_{\alpha}^k e_\alpha^k  \frac{(-1)^k}{k! (h_\alpha+\rho(h_\alpha)+1)\ldots (h_\alpha+\rho(h_\alpha)+k)}. 
$$
Set $N=|\Delta^+|$. Choose a numbering of positive roots, $\alpha_1,\ldots,\alpha_N$.
Each $p_\alpha$, as well as any product of finitely many of them,  is a formal series with coefficients in $\mathbb C(\gt h^*)$ in monomials 
$$
f_{\alpha_1}^{r_1}\ldots f_{\alpha_N}^{r_N} e_{\alpha_N}^{k_N} \ldots e_{\alpha_1}^{k_1}
\ \text{ such that } \ (k_1-r_1)\alpha_1+\ldots + (k_N-r_N)\alpha_N=0. 
$$  
A total order on $\Delta^+$ is said to be  {\it normal}  if 
either $\alpha < \alpha+\beta < \beta $ or $\beta < \alpha+\beta < \alpha$ for each pair of positive roots 
$\alpha, \beta$ such that $\alpha+\beta\in\Delta$.  There is a bijection between the normal orders and
the reduced decompositions of the longest element of $W$.  

Choose a normal order $\alpha_1<\ldots <\alpha_N$, 
and define 
$$
p=p_{\alpha_1}\ldots p_{\alpha_N} 
$$
accordingly. 
The element $p$ is called the {\it extremal projector}.  It is independent of the choice of a normal order. 
For proofs and more details on this operator see \cite[\S 9.1]{book}. Most importantly, it has the property that
\begin{equation}\label{eq-p}
e_\alpha p=p f_{\alpha}=0
\end{equation}
 for each $\alpha$. 

The nilpotent radical $\gt n_\ell\subset \gt b_\ell$ acts on $M_\lambda$ locally 
nilpotently. Recall that $\gt n^+\subset \gt n_\ell$. 
Let $v\in M_\lambda$ be an eigenvector of $\gt h$ of weight $\gamma\in\gt h^*$. 
First of all, $pv$
is a finite sum of vectors of $M_\lambda$ with coefficients in $\mathbb C(\gt h^*)$.   
Second, if all the appearing denominators are non-zero at $\gamma$, then $pv$ is a well-defined 
vector of $M_\lambda$ of the same weight $\gamma$.

\section{Proof of Theorem~\ref{th-eq}}

Let $\lambda$,  $\mu$, and $M_\lambda$ be as in Lemma~\ref{lc}.  Keep in mind that  
$\lambda$ and $\mu$ are arbitrary elements of $\mK^\ell$ and $\mK^n$. 
Since each
${\mathcal R}_i$  commutes with  $\gt g$, 
each ${\mathcal R}_i v_\lambda$ is a highest weight vector of $\gt g$.


We use the extremal projector $p$ associated with $\gt g$. 
If $p$ can be applied to a highest weight vector $v$, then $pv=v$. 
Suppose that $p$  is defined at $\mu$. Then, in view of \eqref{T-t} and \eqref{eq-p},   
$$
{\mathcal R}_i v_\lambda=  p{\mathcal R}_i v_\lambda =  p {\mathcal T}(P_i^{[1]} )|_{t=1} v_\lambda .
$$
Assume that $J(\{P_i^{[1]}\})(\mu) = 0$. Then there is a non-trivial linear combination 
${\mathcal R}=\sum c_i {\mathcal R}_i$
such  that 
$$
{\mathcal R} v_\lambda \in \gt n^-\U(\gt b^-)(\gt n^+ u) v_\lambda, 
$$
see Lemma~\ref{lc}. Here $p {\mathcal R} v_\lambda =0$ and hence 
${\mathcal R} v_\lambda =0$ as well. 

Since we are considering a Verma module of $\gt{gl}_{\ell}(\mK)$, this implies that 
$$
{\mathcal R} \in \U(\gt{gl}_{\ell}(\mK))\gt b_\ell \cap \U(\gt q) = \U(\gt q)\gt b.
$$ 
Hence the symbol ${\rm gr}({\mathcal R})$ of ${\mathcal R}$ lies in the ideal of $\cS(\gt q)$ generated by $\gt b$. 

The decomposition $\gt g=\gt n^-{\oplus}\gt b$ defines a bi-grading on $\cS(\gt g)$.  
Let $H_i^\bullet$ be the bi-homogeneous component of 
$H_i$ having the  highest degree w.r.t. $\gt n^-$. 
According to  \cite[Sec.~3]{py-feig},  $H_i^\bullet\in \gt b\cS^{d_i{-}1}(\gt n^-)$ and
the polynomials $H_1^\bullet,\ldots,H_n^\bullet$ are algebraically independent. 
We have $$
\psi(H_i^\bullet) \in (\gt b u) \cS^{d_i{-}1}(\gt n^-) \oplus \gt b(\gt n^- u)\cS^{d_i{-}2}(\gt n^-).
$$ 
Write this as $\psi(H_i^\bullet)\in H_{i,1} + \gt b(\gt n^- u)\cS^{d_i{-}2}(\gt n^-)$.
Then the polynomials $H_{i,1}$ with $1\le i\le n$ are still algebraically independent. 
As can be easily seen, 
$\psi(H_i)\in H_{i,1} + \gt b \cS(\gt q)$. 

Set $d=\max\limits_{i, c_i\ne 0} d_i$. Then 
$$ {\rm gr}({\mathcal R})=\sum\limits_{i, \, d_i=d} c_i \psi(H_i)$$
and it lies in $(\gt b)\lhd \cS(\gt q)$ if and only if 
$\sum\limits_{i, \, d_i=d} c_i H_{i,1}=0$. Since at least one $c_i$ in this linear combination is 
non-zero, we get a contradiction. The following is settled: 
if $p$ is defined at $\mu$, then 
$J(\{P_i^{[1]}\})(\mu)\ne 0$. 

Now we know that 
the zero set of 
$J(\{P_i^{[1]}\})$ lies in the union of hyperplanes $h_\alpha+\rho(h_\alpha)=-k$ with $k\ge 1$. 
 At the same time this zero set is an affine subvariety of $\mathbb C^n$ of codimension one. 
Therefore it is the union of $N$ hyperplanes and $J(\{P_i^{[1]}\})$ is the product of $N$ linear factors of the form $(h_\alpha+\rho(h_\alpha)+k_\alpha)$. A priory, a root $\alpha$ may appear in several factors with different constants $k_\alpha$.  

By Lemma~\ref{highest}, the highest degree component of 
$J(\{P_i^{[1]}\})$ is equal to $C \prod\limits_{\alpha\in\Delta^+} h_\alpha$.
Therefore each $\alpha\in\Delta^+$ must appear in  exactly one linear factor of  $J(\{P_i^{[1]}\})$.   
Observe that $\rho(h_\alpha)\ge 1$ and that $\rho(h_\alpha)+k_\alpha \ge \rho(h_\alpha)+1$. 
If for some $\alpha$, we have $k_\alpha > 1$, 
then $$|J(\{P_i^{[1]}\})(0)|> |C| \prod\limits_{\alpha\in\Delta^+} (\rho(h_\alpha)+1).$$ But this cannot be the case
in view of Lemma~\ref{free}. 
 \qed

\section{Conclusion} 

The elements ${\mathcal R}_i$ are rather natural $\gt g$-invariants in $\U(\gt q)$ of degree one 
in $\gt g u$. Note that because $\gt g u$ is an Abelian ideal of $\gt q$, where is no ambiguity in defining the 
degree in $\gt g u$. The involvement  of these elements in the branching rules $\gt q\downarrow \gt g$ 
remains unclear. However, combining Lemma~\ref{lc} with Theorem~\ref{th-eq}, we obtain the following statement. 

\begin{cl}\label{cl-c} In the notation of Lemma~\ref{lc}, there is a non-trivial linear combination 
${\mathcal R}=\sum c_i {\mathcal R}_i$ such that 
${\mathcal R} v_\lambda \in \gt n^- \U(\gt q) v_\lambda$ if and only if 
$\mu(h_\alpha)=-\rho(h_\alpha)-1$ for some $\alpha\in\Delta^+$. \qed
\end{cl} 

As the  theory  of  finite-dimensional representations suggests, it is unusual for a highest weight vector of $\gt g$ to
belong to the image of $\gt n^-$.  The proof of Theorem~\ref{th-eq} shows that ${\mathcal R}v_\lambda\ne 0$ 
for the linear combination of Corollary~\ref{cl-c}. 

\begin{rmk}
The subspace ${\mathcal V}[1]=(\U(\gt g)(\gt g u))^{\gt g}\subset \U(\gt q)$ is a $\mathcal{ZU}(\gt g)$-module. 
From a well-known description of $(\gt g{\otimes}{\mathcal S}(\gt g))^{\gt g}$, one can deduce 
that ${\mathcal V}[1]$ is freely generated by 
${\mathcal R}_1,\ldots,{\mathcal R}_n$ as a $\mathcal{ZU}(\gt g)$-module.  
There are other choices of generators in ${\mathcal V}[1]$ and it is not clear, whether one can get nice formulas for 
the corresponding Jacobians.  
\end{rmk}


\begin{thebibliography}{NM96}

\bibitem[AST]{Proj}
{\sc  R. M. Asherova}, {\sc Yu. F. Smirnov}, and   {\sc V. N. Tolstoy},  Projection operators for simple Lie groups, {\it Theor. Math. Phys.} {\bf 8}\,(1971), 813--825 (in Russian).   


\bibitem[Di74]{Dix} 
{\sc J. Dixmier}, 
Alg{\`e}bres enveloppantes. Gauthier-Villars (1974).


\bibitem[G95]{Geof}
{\sc F. Geoffriau},
Homorphisme de Harish-Chandra pour les  alg{\`e}bre de Takiff g{\'e}n{\'e}ralis{\'e}es,
 {\it J. Algebra} {\bf 171}\,(1995), 444--456.


\bibitem[H90]{Ref-b}
{\sc J. E. Humphreys},
Reflection groups and Coxeter groups. Cambridge Studies in Advanced Mathematics, {\bf  29}. 
Cambridge University Press, Cambridge, 1990. 

\bibitem[M07]{book}
{\sc A. Molev}, 
Yangians and Classical Lie Algebras.
Mathematical Surveys and Monographs, {\bf 143},  American Mathematical Society, Providence, RI, 2007; 
Russian edition: MCCME Moscow 2009.

\bibitem[PY12]{py-feig}
{\sc D. I. Panyushev} and {\sc  O. S. Yakimova},
A remarkable contraction of semisimple Lie algebras, {\it Ann. Inst. Fourier (Grenoble)}, {\bf 62} no.\,6 (2012), 2053--2068.

\bibitem[PY]{k-T}
{\sc D. I. Panyushev} and {\sc  O. S. Yakimova}, Takiff algebras with polynomial rings of symmetric invariants, 
{\tt  http://arxiv.org/abs/1710.03180}.

\bibitem[Ta71]{takiff} 
{\sc S. J. Takiff}, 
Rings of invariant polynomials for a class of Lie algebras, 
{\it Trans. Amer. Math. Soc.} {\bf 160}\,(1971), 249--262.


\bibitem[W11]{Wil}
{\sc B. J. Wilson}, 
Highest-weight theory for truncated current Lie algebras, {\it J. Algebra}, {\bf 336}\,(2011), 1--27.


\end{thebibliography}
\end{document}